\documentclass[a4paper,12pt,thmsa]{amsart}
\usepackage[a4paper,marginratio={1:1},scale={0.72,0.74},footskip=7mm,headsep=10mm]{geometry}

\usepackage{amsfonts}
\usepackage{amssymb,amsmath,latexsym}
\usepackage[dvips]{graphics}
\usepackage{graphicx,subfigure}
\usepackage[dvips]{color}
\usepackage[T1]{fontenc}
\usepackage[active]{srcltx}
\usepackage{amsmath}
\usepackage{amsfonts}
\usepackage{amssymb}
\usepackage{psfrag}
\usepackage{color}
\usepackage{url}
\usepackage{amsthm}
\usepackage{array}
\usepackage{mathrsfs}

\usepackage{pst-tree}
\usepackage{lscape}
\usepackage[T1]{fontenc}
\usepackage{pstricks,pstricks-add}

\newtheorem{thm}{Theorem}[section]
\newtheorem{proposition}[thm]{Proposition}
\newtheorem{remark}[thm]{Remark}
\newtheorem{definition}[thm]{Definition}
\newtheorem{corollary}[thm]{Corollary}

\newtheorem{lemma}[thm]{Lemma}

\newcommand{\ed}{\stackrel{\mbox{\tiny $(d)$}}{=}}

\def\qed{\hfill $\square$ \bigskip}

\def\beq{\begin{equation}}               
\def\eeq{\end{equation}}                 
\def\bea{\begin{eqnarray}}             
\def\eea{\end{eqnarray}}               
\def\be*{\begin{eqnarray*}}             
\def\ee*{\end{eqnarray*}}               
\def\ba{\begin{array}}                  
\def\ea{\end{array}}                    

\def\beqlb{\begin{eqnarray}} \def\eeqlb{\end{eqnarray}}
\def\beqnn{\begin{eqnarray*}} \def\eeqnn{\end{eqnarray*}}

\def\<{\langle}  \def\>{\rangle}

\def\bde{\begin{definition}}
\def\ede{\end{definition}}

\def\bth{\begin{thm}}
\def\eth{\end{thm}}
\def\bpr{\begin{proposition}}
\def\epr{\end{proposition}}
\def\ble{\begin{lemma}}
\def\ele{\end{lemma}}
\def\bcor{\begin{corollary}}
\def\ecor{\end{corollary}}
\def\bre{\begin{remark}}
\def\ere{\end{remark}}

\def\p{\mathbb{P}}
\def\ee{\varepsilon}
\def\qed{{\hfill $\Box$ \bigskip}}

\def\qed{{\hfill $\Box$ \bigskip}}

\title[Breadth first search coding and Lamperti representation]
{Breadth first search coding of multitype forests with application to Lamperti representation}

\author{Lo\"ic Chaumont}

\address{Lo\"ic Chaumont -- LAREMA -- UMR CNRS 6093, Universit\'e d'Angers, 2 bd Lavoisier, 49045 Angers cedex~01}

\email{loic.chaumont@univ-angers.fr}



\date{\today}

\begin{document}


\maketitle

\begin{center}
Dedicated to the memory of Marc Yor
\end{center}

\vspace*{.3in}

\begin{abstract} We obtain a bijection between some set of multidimensional sequences and this of $d$-type plane forests which is
based on the breadth first search algorithm. This coding sequence is related to the sequence of population sizes indexed by the generations, through a Lamperti type transformation. The same transformation in then obtained in continuous time for multitype branching processes with discrete values. We show that any such process can be obtained from a $d^2$ dimensional compound
Poisson process time changed by some integral functional. Our proof bears on the discretisation of branching  forests with edge 
lengths.
\end{abstract}

\vspace*{.3in}

\section{Introduction}\label{intro}

A famous result from Lamperti \cite{la} asserts that any continuous state branching process can be represented as a spectrally 
positive L\'evy process, time changed by the inverse of some integral functional. This transformation is invertible and defines a bijection between the set of spectrally positive L\'evy processes and this of continuous  state branching processes. Lamperti's 
result is the source of an extensive mathematical literature in which it is mainly used as a tool in branching theory. However,
recently Lamperti's transformation itself has been the focus of some research papers. In \cite{clu} several proofs of this result
are displayed and in \cite{cpu} an extension of the transformation to the case of branching processes with immigration is 
obtained. 

In this work we show an extension of Lamperti's transformation to continuous time, integer valued, multitype branching processes. 
More specifically, let $Z=(Z^{(1)},\dots,Z^{(d)})$ be such a process issued from $x\in\mathbb{Z}_+^d$, then we shall prove that 
\[(Z^{(1)}_t,\dots,Z^{(d)}_t)=x+\left(\sum_{i=1}^dX^{i,1}_{\int_0^t Z^{(i)}_s\,ds},\dots,\sum_{i=1}^dX^{i,d}_{\int_0^t Z^{(i)}_s\,ds}\right)
\,,\;\;\;t\ge0\,,\]
where $X^{(i)}$ are $d$ independent $\mathbb{Z}_+^d$-valued compound Poisson processes. As in the one dimensional case, absorption of $Z$ at 0 means that in this transformation, the process $X$ must be stopped
at some random time which, in the multitype case, will be defined as the `first passage time' at level $-x$ by
the multidimensional random field 
\[{\rm X}=(t_1,\dots,t_d)\mapsto\left(\sum_{i=1}^dX_{t_i}^{i,j},\,j\in[d]\right)=
X^{(1)}_{t_1}+\dots+X^{(d)}_{t_d}\,.\]
Multitype Lamperti transformation is not invertible as in the one dimensional case. However, by considering
the whole branching structure behind the branching process, it is possible, to obtain a one-to-one relationship
between $(X^{(1)},\dots,X^{(d)})$ and some decomposition of the $d$ dimensional process $Z$ into $d^2$
processes, see Theorem \ref{2379}.       

The proofs of these results pass through a special coding of multitype plane forests which leads to a Lamperti type representation of discrete time, multitype branching processes. Results in discrete time are displayed 
and proved in Section \ref{deterministic} whereas the next section is devoted to the statements of our results 
in continuous time. The  latter will be proved in Section~\ref{proofs}.  

\section{Main results in continuous time}\label{main}

In all this work, we use the notation $\mathbb{R}_+=[0,\infty)$, $\mathbb{Z}_+=\{0,1,2,\dots\}$ and for any positive integer $d$, we set $[d]=\{1,\dots,d\}$. We also define the following partial order on $\mathbb{R}^d$
by setting $x=(x_1,\dots,x_d)\ge y=(y_1,\dots,y_d)$, if $x_i\ge y_i$, for all $i\in [d]$. Moreover, we write 
$x>y$ if $x\ge y$ and if there exists $i\in[d]$ such that $x_i>y_i$. We will denote by $e_i$ the $i$-th unit 
vector of $\mathbb{Z}_+^d$.\\

Fix $d\ge2$, let $\nu=(\nu_1,\dots,\nu_d)$, where $\nu_i$ is any probability measure on $\mathbb{Z}_+^d$ such that $\nu_i(e_i)<1$
and let $Z=(Z^{(1)},\dots,Z^{(d)})$ be a $d$-type continuous time and $\mathbb{Z}_+^d$-valued branching process with progeny distribution $\nu=(\nu_1,\dots,\nu_d)$ and such that type $i\in[d]$ has reproduction 
rate $\lambda_i>0$. For $i,j\in [d]$, the quantity
\[m_{ij}=\sum_{x\in\mathbb Z_+^d}x_j\nu_i(x)\,,\]
 corresponds to the mean number of children of type $j$, given by an individual of type $i$. Let $M:=(m_{ij})_{i,j\in [d]}$
be the mean matrix of $Z$. Recall that if $M$ (or equivalently $\nu$) is irreducible, then according to 
Perron-Frobenius Theorem, it admits a unique eigenvalue $\rho$ which is simple,
 positive and with maximal modulus. If moreover, $\nu$ is non degenerate, then extinction holds if and only if $\rho\le1$, see \cite{ha1}, \cite{mo} and Chapter V of \cite{an}. If $\rho=1$, we say that $Z$ is critical and if $\rho<1$, we say that $Z$ 
 is subcritical.\\

We now define the underlying compound Poisson process in the Lamperti representation of $Z$ that will be presented in
Theorems \ref{2379} and \ref{2415}.  Let $X=(X^{(1)},\dots,X^{(d)})$, where $X^{(i)}$, $i\in[d]$ are $d$ independent 
$\mathbb{Z}^d$-valued compound Poisson 
processes. We assume that $X^{(i)}_0=0$ and that $X^{(i)}$ has rate $\lambda_i$ and jump distribution 
\begin{equation}\label{7458}
\mu_i({\rm k})=\frac{\tilde{\nu}_i({\rm k})}{1-\tilde{\nu}_i(0)}\,,\;\;\mbox{if ${\rm k}\neq0$ and}\;\;\;\mu_i(0)=0\,,
\end{equation}
where 
\begin{equation}\label{4590}
\tilde{\nu}_i(k_1,\dots,k_d)=\nu_i(k_{1},\dots,k_{i-1},k_{i}+1,k_{i+1},\dots,k_{d})\,.
\end{equation}
In particular, with the notation $X^{(i)}=(X^{i,1},\dots,X^{i,d})$, for all $i=1,\dots,d$, the process $X^{i,i}$ 
is a $\mathbb{Z}$-valued, downward skip free, compound Poisson process, i.e. 
$\Delta X_t^{i,i}=X^{i,i}_t-X^{i,i}_{t-}\ge-1$, $t\ge0$, with $X_{0-}=0$ and for all $i\neq j$, the process 
$X^{i,j}$ is a standard Poisson process. We emphasize that in this definition, some of the processes 
$X^{i,j}$, $i,j\in[d]$ can be identically equal to 0.\\

We first present a result on passage times of the multidimensional random field 
\[{\rm X}:(t_1,\dots,t_d)\mapsto\left(\sum_{i=1}^dX_{t_i}^{i,j},\,j\in[d]\right)=
X^{(1)}_{t_1}+\dots+X^{(d)}_{t_d}\,,\]
which is a particular case of additive L\'evy processes, see \cite{kx} and the references therein.  
Henceforth, a process such as ${\rm X}$ will be called an {\it additive $($downward skip free$)$ compound 
Poisson process}. 

\begin{thm}\label{6803} Let $x=(x_1,\dots,x_d)\in\mathbb{Z}_+^d$. Then there exists a $($unique$)$ 
random time $T_x=(T_x^{(1)}\dots,T_x^{(d)})\in\overline{\mathbb{R}}_+^d$ such that almost surely,
\begin{equation}\label{3678}
x_j+\sum_{i,T^{(i)}_x<\infty}X^{i,j}(T_x^{(i)})=0\,,\;\;\;\mbox{for all $j$ such that $T^{(j)}_x<\infty$}\,,
\end{equation}
and if $T'_x$ is any random time satisfying $(\ref{3678})$, then  $T'_x\ge T_x$. The time $T_x$ will be called 
the first passage time of the additive compound Poisson process $X$ at level $-x$. 

The process $(T_x,x\in\mathbb{Z}_+^d)$ is increasing and additive, that is, for all $x,y\in\mathbb{Z}_+^d$,
\begin{equation}\label{6290}T_{x+y}\ed T_x+\tilde{T}_y\,,
\end{equation}
where $\tilde{T}_y$ is an independent copy of $T_y$.
The law of $T_x$ on $\mathbb{R}_+^d$ is given by 
\begin{eqnarray*}
&&\p(T_x\in dt,\,X^{i,j}_{t_i}=x_{i,j},\,1\le i,j\le d)=\\
&&{}\\
&&\qquad\frac{\mbox{det}(-x_{i,j})}{t_1t_2\dots t_d}\prod_{i=1}^d\p(X^{i,j}_{t_i}=x_{i,j},\,1\le j\le d)dt_1dt_2\dots dt_d\,, 
\end{eqnarray*}
where the support of this measure is included in
$\left\{x_{ij}\in\mathbb{Z}:x_{ij}\ge0,\,x_{ii}\le0,\, \sum_{i=1}^d x_{i,j}=-x_i \right\}$.
\end{thm}
\noindent 
Note that from the additivity property (\ref{6290}) of $(T_x,\,x\in\mathbb{Z}_+^d)$, we derive that the law of
this process is characterised by the law of the variables $T_{e_i}$ for $i\in[d]$.\\

\noindent As the above statement suggests, some coordinates of the time $T_x$ may be infinite. More specifically, we have:

\begin{proposition}\label{1269}  Assume that $\nu$ irreducible and non degenerate.
\begin{itemize}
\item[$1.$] If $\nu$ is $($sub$)$critical, then almost surely, for all $x\in\mathbb{Z}_+^d$ and for all $i\in[d]$, $T_x^{(i)}<\infty$.
\item[$2.$] If $\nu$ is super critical, then for all $x\in\mathbb{Z}_+^d$, with some probability $p>0$, $T_x^{(i)}=\infty$, for all 
$i\in[d]$ and  with probability $1-p$,  $T_x^{(i)}<\infty$, for all $i\in[d]$.
\end{itemize}
\end{proposition}
\noindent There are instances of reducible distributions $\nu$ such that for some $x\in\mathbb{Z}_+^d$,  with positive probability, $T_x^{(i)}<\infty$, for all $i\in A$ and $T_x^{(i)}=\infty$, for all $i\in B$,  $(A,B)$ being a non trivial partition of $[d]$. It is the case for instance when $d=2$, for $x=(1,1)$, $X^{1,2}=X^{2,1}\equiv 0$, $0<m_{11}<1$ and $m_{22}>1$.\\

Then we define $d$-type branching forests with edge lengths as finite 
sets of independent branching trees with edge lengths. We say that such a forest is issued from $x=(x_1,\dots,x_d)\in\mathbb{Z}_+^d$ (at time $t=0$), if it contains $x_i$ trees whose root is of type $i$. The discrete
skeleton of a branching forest with edge lengths is a discrete branching forest with progeny distribution $\nu$. The edges issued 
from vertices of type $i$ are exponential random variables with parameter $\lambda_i$. These random variables are independent between themselves and are independent of the discrete skeleton. A realisation of such a forest is represented in Figure \ref{fig1}. Then to each $d$-type forest with edge lengths $F$
is associated the branching process $Z=(Z^{(1)},\dots,Z^{(d)})$, where $Z^{(i)}$ is the number of individuals 
in $F$, alive at time $t$.  

\begin{definition}\label{2593}
For $i\neq j$, we denote by $Z^{i,j}_t$ the total number of individuals of type $j$ whose parent has type $i$ and 
who were born before time $t$. The definition of $Z^{i,i}_t$ for $i\in[d]$
is the same, except that we add the number of roots of type $i$ 
and we subtract the number of individuals of type $i$ who died 
before time $t$. 
\end{definition}

Then we readily check that the branching process $Z=(Z_1,\dots,Z_d)$ which is associated to this forest can 
be expressed in terms of the processes $Z^{i,j}$, as follows:
\[Z^{(j)}_t=\sum_{i=1}^dZ^{i,j}_t\,,\;\;\;j\in[d]\,.\]

\begin{figure}[hbtp]

\vspace*{-.3cm}

\hspace*{-0.5cm}{\includegraphics[height=350pt,width=500pt]{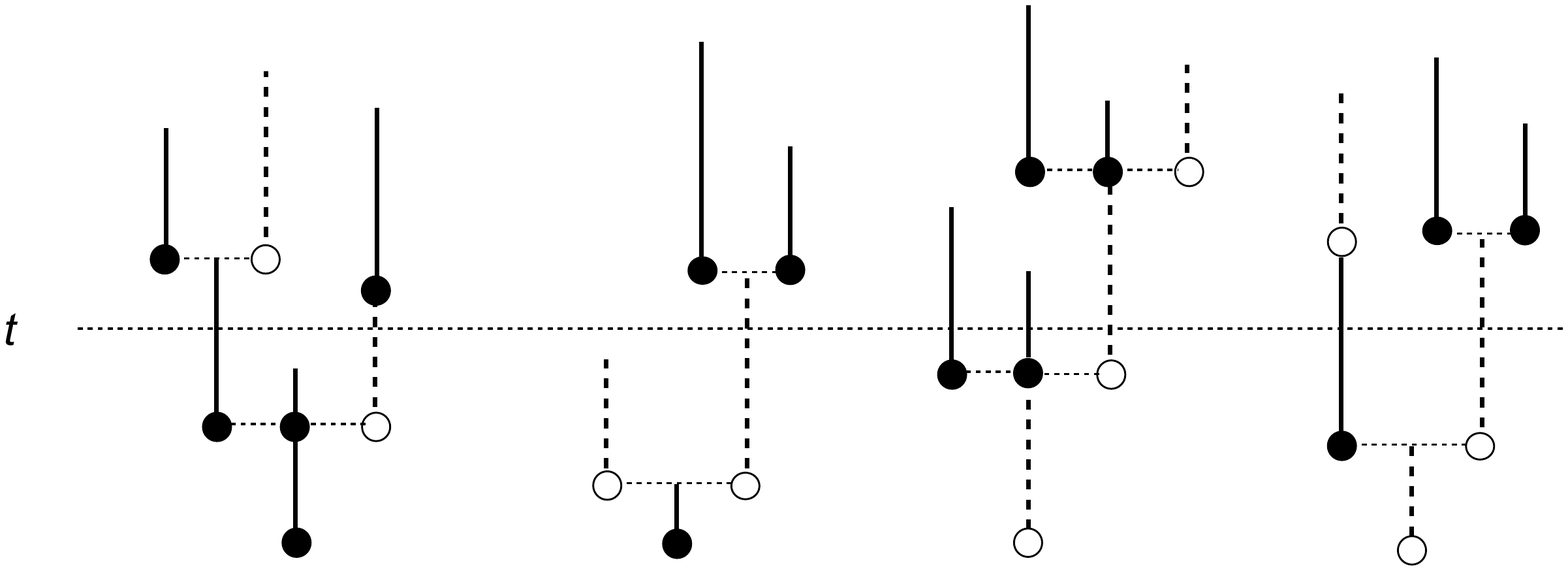}}
 
\vspace{-6.5cm}

\caption{A two type forest with edge length issued from $x=(2,2)$. Vertices of type 1 (resp.~2) are represented in black 
(resp.~white). At time $t$, $Z^{1,1}_t=2$, $Z^{1,2}_t=2$, $Z^{2,1}_t=3$ and $Z^{2,2}_t=4$.}\label{fig1}
\end{figure}

\noindent The next theorem asserts that from a given $d$-type forest with edge lengths $F$, issued from 
$x=(x_1,\dots,x_d)$, we can construct a $d$-dimensional additive compound Poisson process
${\rm X}=(\sum_{i=1}^dX_{t_i}^{i,j},\,j\in[d],\,(t_1,\dots,t_d)\in\mathbb{R}_+^d)$ stopped at its first passage time 
of $-x$, such that the branching process $Z$ associated to $F$ can be represented as a time change 
of ${\rm X}$. This extends the Lamperti representation to multitype branching processes.

\begin{thm}\label{2379} Let $x=(x_1,\dots,x_d)\in\mathbb{Z}_+^d$ and let $F$ be a $d$-type branching forest
with edge lengths, issued from $x$, with progeny distribution $\nu$ and life time rates $\lambda_i$. Then the processes 
$Z^{i,j}$, $i,j\in[d]$ introduced in Definition $\ref{2593}$ admit the following representation:
\begin{equation}\label{3521}
Z^{i,j}_t=\left\{\begin{array}{ll}
X^{i,j}_{\int_0^t Z^{(i)}_s\,ds}\,,\;\;\;t\ge0\,,\;\;\;\mbox{if $i\neq j$,}\\
x_i+X^{i,i}_{\int_0^t Z^{(i)}_s\,ds},\;\;\;t\ge0\,,\;\;\;\mbox{if $i=j$,}
\end{array}\right.
\end{equation}
where 
\[X^{(i)}=(X^{i,1},X^{i,2},\dots,X^{i,d})\,,\;\;\;i=1,\dots,d\,,\]
are independent $\mathbb{Z}_+^d$-valued compound Poisson processes with respective rates $\lambda_i$ and
jump distributions $\mu_i$ defined in $(\ref{7458})$, stopped at the first hitting time $T_x$ of $-x$ by the additive compound 
Poisson process, ${\rm X}=\left(\sum_{i=1}^dX_{t_i}^{i,j},\,j\in[d],\,(t_1,\dots,t_d)\in\mathbb{R}_+^d\right)$, i.e.
\[X^{(i)}_t{\bf 1}_{\{t<T_x^{(i)}\}}+(X_{T_x^{(i)}}^{i,1},\dots,X_{T_x^{(i)}}^{i,d}){\bf 1}_{\{t\ge T_x^{(i)}\}}\,.\]

In particular the multitype branching process $Z$, issued from $x=(x_1,\dots,x_d)\in\mathbb{R}_+^d$ admits the following representation, 
\begin{equation}\label{2415}
(Z^{(1)}_t,\dots,Z^{(d)}_t)=x+\left(\sum_{i=1}^dX^{i,1}_{\int_0^t Z^{(i)}_s\,ds},\dots,\sum_{i=1}^dX^{i,d}_{\int_0^t Z^{(i)}_s\,ds}\right)
\,,\;\;\;t\ge0\,.
\end{equation}
Moreover, the transformation $(\ref{3521})$ is invertible, so that the processes $Z^{i,j}$ can be recovered 
from the processes $X^{(i)}$. 
\end{thm}

\noindent Conversely, the following theorem asserts that an additive compound Poisson process being given, 
we can construct a multitype branching forest whose branching process is the unique solution of equation
(\ref{2415}). 

\begin{thm}\label{7462} Let $x=(x_1,\dots,x_d)\in\mathbb{R}_+^d$ and
\[X^{(i)}=(X^{i,1},X^{i,2},\dots,X^{i,d})\,,\;\;\;i=1,\dots,d\,,\]
be independent $\mathbb{Z}_+^d$ valued compound Poisson processes with respective rates $\lambda_i>0$ and jump 
distributions $\mu_i$, stopped at the first hitting time $T_x$ of $-x$ by the additive compound Poisson process  $(t_1,\dots,t_d)\mapsto\left(\sum_{i=1}^dX_{t_i}^{i,j},\,j\in[d]\right)$. Then there is a branching forest with
edge lengths, with progeny distribution $\nu$ and rates $\lambda_i>0$
such that the processes $Z^{i,j}$ of Definition $\ref{2593}$ satisfy relation $(\ref{3521})$.
Moreover, the branching process $Z=(Z^{(1)},\dots,Z^{(d)})$ associated to this forest is the unique solution of 
the equation, 
\[(Z^{(1)}_t,\dots,Z^{(d)}_t)=x+\left(\sum_{i=1}^dX^{i,1}_{\int_0^t Z^{(i)}_s\,ds},\dots,
\sum_{i=1}^dX^{i,d}_{\int_0^t Z^{(i)}_s\,ds}\right)\,,\;\;\;t\ge0\,.\]
\end{thm}

\noindent We emphasize that Theorems \ref{2379} and \ref{7462} do not define a bijection between the set of  branching forests 
with edge lengths and this of additive compound Poisson processes, as in the discrete time
case, see Section \ref{deterministic}. Indeed, in the continuous time case, when constructing the processes
$Z^{i,j}$ as in  Definition \ref{2593}, at each birth time, we loose the information of the specific individual who
gives  birth. In particular, the forest which is constructed in Theorem \ref{7462} is not unique. 
This lost information is preserved in discrete time and the breadth first search coding that is defined 
in Subsection \ref{codingforests} allows us to define a bijection between the set of discrete forests and this of
coding sequences.

\section{Discrete multitype forests}\label{deterministic}

\subsection{The space of multitype forests}\label{space} We will denote by 
$\mathscr{F}$ the set of plane forests. More specifically, an element ${\bf f}\in\mathscr{F}$ is a directed planar 
graph with no loops on a possibly infinite and non empty set of vertices ${\bf v} = {\bf v} ({\bf f})$, with  a {\it finite} number of connected components, such that each vertex has a finite inner degree and
an outer degree equals to 0 or 1. The elements of $\mathscr{F}$ 
will simply be called forests.The connected components of a forest are called the {\it trees}. A forest consisting of a single connected component is also called a tree. 
In a tree ${\bf t}$, the only vertex with outer degree equal to 0 is called the {\it root} of ${\bf t}$. It will be denoted by 
$r({\bf t})$. The roots of the connected components of a forest ${\bf f}$ are called the roots of ${\bf f}$. For two vertices 
$u$ and $v$ of a forest ${\bf f}$, if $(u,v)$ is a directed edge of ${\bf f}$, then we say that $u$ is a {\it child} of $v$, or 
that $v$  is the {\it parent} of $u$. The set ${\bf v}({\bf f})$ of vertices of each forest ${\bf f}$ will be enumerated according to the 
usual breadth first search order, see Figure \ref{fig2}. We emphasize that we begin by enumerating the roots of the forest from 
the left to the right. In particular, our enumeration is not performed tree by tree. If a forest ${\bf f}$ contains at least $n$ 
vertices, then the $n$-th vertex of ${\bf f}$ is denoted by $u_n({\bf f})$. When no confusion is possible, we will simply
denote the $n$-th vertex by $u_n$.\\

A $d$-type forest is a couple $({\bf f},c_{\bf f})$, where ${\bf f}\in\mathscr{F}$ and $c_{\bf f}$ is an application 
$c_{\bf f}:{\bf v}({\bf f})\rightarrow [d]$. For $v\in{\bf v}({\bf f})$, the integer $c_{\bf f}(v)$ is called the {\it type} (or the {\it
color}) of $v$. The set of finite $d$-type forests will be denoted by $\mathscr{F}_d$. An element $({\bf f},c_{\bf f})\in
\mathscr{F}_d$ will often simply be denoted by ${\bf f}$.  We assume that for any  ${\bf f}\in\mathscr{F}_d$, if $u_i,u_{i+1},\dots,u_{i+j}\in{\bf v}({\bf f})$ have the same parent, then $c_{\bf f}(u_i)\le c_{\bf f}(u_{i+1})\le\dots\le c_{\bf f}(u_{i+j})$. Moreover, 
if $u_1,\dots,u_k$ are the roots of ${\bf f}$, then $c_{\bf f}(u_1)\le\dots\le c_{\bf f}(u_{k})$.  For each $i\in[d]$ we will denote by $u^{(i)}_n=u^{(i)}_n({\bf f})$ the $n$-th vertex of type $i$ of the forest ${\bf f}\in\mathscr{F}_d$, see Figure \ref{fig2}.\\

\begin{figure}[hbtp]

\vspace*{-.5cm}

\hspace*{-0.5cm}{\includegraphics[height=350pt,width=500pt]{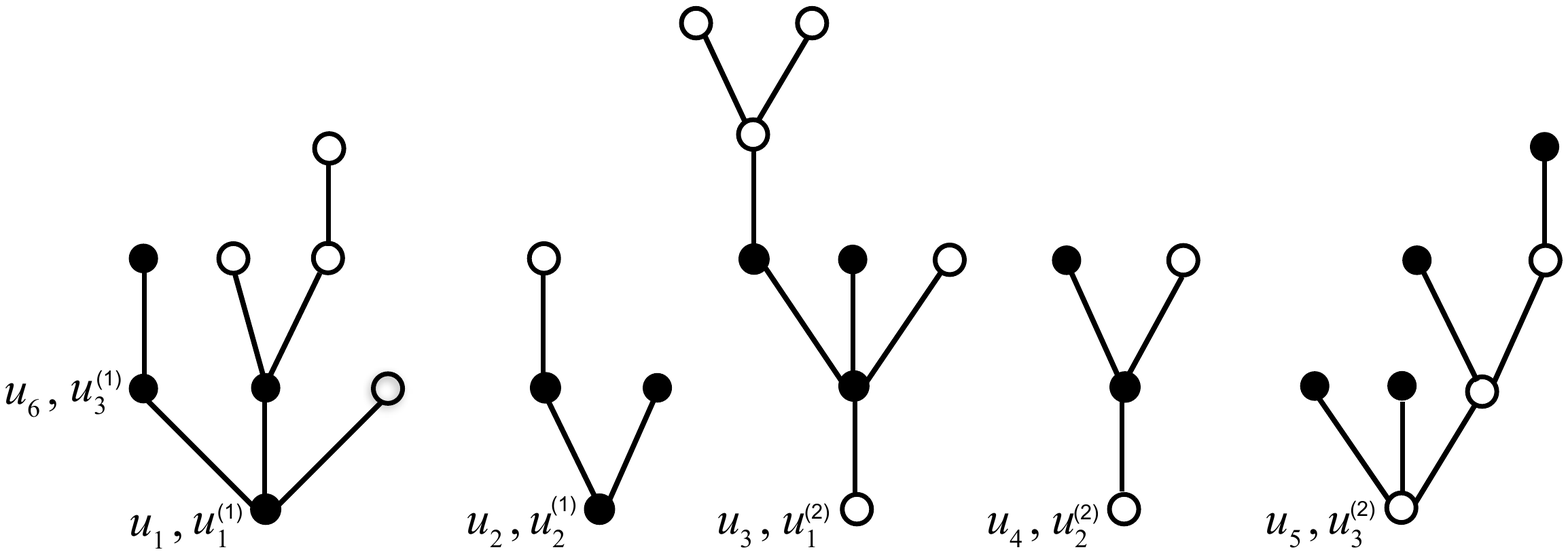}}
 
\vspace{-6cm}

\caption{A two type forest labeled according to the breath first search order. Vertices of type 1 (resp.~2) are represented in white (resp.~black).}\label{fig2}
\end{figure}

\subsection{Coding multitype forests}\label{codingforests} The aim of this subsection is to obtain a bijection between the 
set of multitype forests and some particular set of  integer valued sequences which has been introduced
in \cite{cl}. This bijection, which will be called a {\it coding}, depends on the breadth first search ordering defined at the previous
subsection. We emphasize that this coding is quite  different  from the one which is defined in \cite{cl}. 

\begin{definition}\label{9076}
Let $S_d$ be the set of $\left[\mathbb{Z}^{d}\right]^d$-valued sequences, $x=(x^{(1)},x^{(2)},\dots,x^{(d)})$, such that
for all $i\in [d]$, $x^{(i)}=(x^{i,1},\dots,x^{i,d})$ is a $\mathbb{Z}^{d}$-valued sequence defined on some
interval of integers, $\{0,1,2,\ldots, n_i\}$, $0\le n_i\le\infty$, which satisfies $x^{(i)}_0=0$ and if $n_i\ge1$, then
\begin{itemize}
\item[$(i)$]  for $i\neq j$, the sequence $(x_n^{i,j})_{0\le n\le n_i}$ is nondecreasing,
\item[$(ii)$] for all $i$, $x_{n+1}^{i,i}-x_n^{i,i}\ge -1$, $0\le n\le n_i-1$.
\end{itemize}
A sequence $x\in S_d$ will sometimes be denoted by $x=(x^{i,j}_k,\,0\le k\le n_i,\,i,j\in [d])$ and for more convenience, we will sometimes denote $x_k^{i,j}$ by $x^{i,j}(k)$. The vector
${\rm n}=(n_1,\dots,n_d)\in\overline{\mathbb{Z}}_+^d$, where $\overline{\mathbb{Z}}_+=\mathbb{Z}_+\cup\{+\infty\}$ will be called 
the length of $x$.
\end{definition}

\noindent Recall the definition of the order on $\mathbb{R}^d$, from the beginning of Section \ref{main} and 
let us set $U_{\rm s}=\{i\in[d]:s_i<\infty\}$, for any ${\rm s}\in\overline{\mathbb{Z}}_+^d$. Then the next lemma
extends Lemma 2.2 in \cite{cl} to the case where the smallest solution of a system such as $({\rm r},x)$ in
(\ref{4379}) may have infinite coordinates.

\begin{lemma}\label{1377} Let $x\in S_d$ whose length 
${\rm n}=(n_1,\dots,n_d)$ is such that $n_i=\infty$ for all $i$ $($i.e. $U_{\rm n}=\emptyset$$)$  and let ${\rm r}=(r_1,\dots,r_d)\in\mathbb{Z}_+^d$. Then there exists 
${\rm s}=(s_1,\dots,s_d)\in\overline{\mathbb{Z}}_+^d$ such that
\begin{equation}\label{4379}
({\rm r},x)\qquad r_j+\sum_{i=1}^dx^{i,j}(s_i)=0\,,\;\;\;j\in U_{\rm s}\,,
\end{equation}
$($we will say that ${\rm s}$ is a solution of the system $(${\rm r},x$))$  and such that any other solution ${\rm q}$ of $({\rm r},x)$ 
satisfies ${\rm q}\ge {\rm s}$.  Moreover we have $s_i=\min\{q:x^{i,i}_q=\min_{0\le k\le s_i}x^{i,i}_k\}$, for all $i\in U_{\rm s}$. 

The solution ${\rm s}$ will be called {\em the smallest solution} of the system $({\rm r},x)$. We emphasize that in $(\ref{4379})$, we may have $U_{\rm s}=\emptyset$, so that according to this definition, the smallest solution of the system $({\rm r},x)$ may be infinite, that is $s_i=\infty$ for all $i\in[d]$. Note also that in $(\ref{4379})$ it is implicit that
$\sum_{i\in [d]\setminus U_{\rm s}}x^{i,j}(s_i)<\infty$, for all $j\in U_{\rm s}$.
\end{lemma}
\begin{proof} This proof is based on the simple observation that for fixed $j$, when at least one of the indices
 $k_i$'s for $i\neq j$ increases, the term $\sum_{i\neq j}x^{i,j}(k_i)$ may only increase and when $k_j$ increases, the term $x^{j,j}(k_j)$ may decrease only by jumps of amplitude $-1$.\\

First recall the notation $U_{\rm s}$, for ${\rm s}\in\overline{\mathbb{Z}}_+^d$ introduced before Lemma \ref{1377} and set $v_j^{(1)}=r_j$ and for $n\ge1$,
\[k_j^{(n)}=\inf\{k:x^{j,j}_k=-v_j^{(n)}\}\;\;\mbox{and}\;\;v_{j}^{(n+1)}=r_j+\sum_{i\neq j}x^{i,j}(k^{(n)}_i)\,,\]
where $\inf\emptyset=\infty$.  Set also ${\rm k}^{(0)}=0$ and $U_{{\rm k}^{(0)}}=[d]$. Then note that since for
$i\neq j$, the $x^{i,j}$'s are positive and increasing, we have 
\[{\rm k}^{(n)}\le {\rm k}^{(n+1)}\;\;\mbox{and}\;\;U_{{\rm k}^{(n+1)}}\subseteq U_{{\rm k}^{(n)}}\,,\;\;\;n\ge0\,.\]
Moreover, for each $n\ge0$, 
\begin{equation}\label{7399}
r_j+\sum_{i\neq j}x^{i,j}(k_i^{(n)})+x^{j,j}(k_j^{(n)})\ge0\,,\;\;\;j\in U_{{\rm k}^{(n)}}\,.
\end{equation}
Define 
\[n_0=\inf\big\{n\ge1:r_j+\sum_{i\neq j}x^{i,j}(k_i^{(n)})+x^{j,j}(k_j^{(n)})=0\,,\;j\in U_{{\rm k}^{(n)}}\big\}\,,\]
where in this definition, we consider that $r_j+\sum_{i\neq j}x^{i,j}(k_i^{(n)})+x^{j,j}(k_j^{(n)})=0$ is satisfied for all $j\in U_{{\rm k}^{(n)}}$ if $U_{{\rm k}^{(n)}}=\emptyset$. Note that ${\rm k}^{(n)}={\rm k}^{(n_0)}$ and $U_{{\rm k}^{(n)}}=U_{{\rm k}^{(n_0)}}$,
for all $n\ge n_0$. The index $n_0$ can be infinite and in general,
we have ${\rm k}^{(n_0)}=\lim_{n\rightarrow\infty}{\rm k}^{(n)}$.
Then the smallest solution of the system  $({\rm r},x)$ in the sense which is defined in Lemma \ref{1377} is ${\rm k}^{(n_0)}$.\\

Indeed, (\ref{4379}) is clearly satisfied for ${\rm s}={\rm k}^{(n_0)}$. Then let ${\rm q}\in\overline{\mathbb{Z}}_+^d$ satisfying (\ref{4379}), that is
\begin{equation}\label{5389}
r_j+\sum_{i\neq j}x^{i,j}(q_i)+x^{j,j}(q_j)=0\,,\;j\in U_{{\rm q}}\,.
\end{equation}
We can prove by induction that ${\rm q}\ge {\rm k}^{(n)}$, for all $n\ge1$. Firstly for (\ref{5389}) to be satisfied, we should have 
$q_j\ge \inf\{k:x^{j,j}(k)=-r_j\}$, for all $j\in U_{{\rm q}}$,
hence ${\rm q}\ge {\rm k}^{(1)}$. Now assume that ${\rm q}\ge {\rm k}^{(n)}$. Then $U_{{\rm q}}\subseteq U_{{\rm k}^{(n)}}$ and from (\ref{7399}) and (\ref{5389}) for each $j\in U_{{\rm q}}$, $q_j\ge \inf\{k:x^{j,j}(k)=-(r_j+\sum_{i\neq j}x^{i,j}(k_i^{(n)}))\}$, hence ${\rm q}\ge {\rm k}^{(n+1)}$.\\

Finally the fact that $s_i=\min\{q:x^{i,i}_q=\min_{0\le k\le s_i}x^{i,i}_k\}$, for all $i\in U_{\rm s}$ readily follows from the above construction of $s_i$.~\qed
\end{proof}

Let $({\bf f},c_{\bf f})\in\mathscr{F}_d$, $u\in{\bf v}({\bf f})$ and denote by $p_i(u)$ the number of children of type $i$ of 
$u$. For each $i\in [d]$, let $n_i\ge0$ be the number of vertices of type $i$ in ${\bf v}({\bf f})$. Then we define the
application $\Psi$ from $\mathscr{F}_d$ to $S_d$ by
\begin{eqnarray}
\Psi(({\bf f},c_{\bf f}))=x\,,\label{7288}
\end{eqnarray}
where $x=(x^{(1)},\dots,x^{(d)})$ and for all $i\in[d]$, $x^{(i)}$ is the $d$-dimensional chain $x^{(i)}=(x^{i,1},\dots,x^{i,d})$, 
with length $n_i$, whose values belong to the set $\mathbb{Z}^{d}$, such that $x_0^{(i)}=0$ and if $n_i\ge1$, then
\begin{equation}\label{6245}
x_{n+1}^{i,j}-x_n^{i,j}=p_j(u_{n+1}^{(i)})\,,\;\;\mbox{if $i\neq j\;$ and}\;\;\;x_{n+1}^{i,i}-x_n^{i,i}=
p_i(u_{n+1}^{(i)})-1\,,\;\;\;0\le n\le n_i-1\,.
\end{equation}
We recall that $u^{(i)}_n$ is the $n$-th vertex of type $i$ in the breadth first search order of ${\bf f}$.
We will see in Theorem \ref{5678} that $(n_1,\dots,n_d)$ is actually the smallest solution of the system $({\rm r},x)$, 
where $r_i$ is the number of roots of type $i$ of the forest ${\bf f}$. This leads us to the following definition.

\begin{definition}\label{9076} Fix ${\rm r}=(r_1,\dots,r_d)\in\mathbb{Z}_+^d$, such that  ${\rm r}>0$.
\begin{itemize}
\item[$(i)$] We denote by  $\Sigma_d^{\rm r}$ the subset of $S_d$ of sequences $x$ with length ${\rm n}=(n_1,\dots,n_d)\in\overline{\mathbb{Z}}_+^d$ such that  ${\rm n}$ is the smallest solution of the system  $({\rm r},\bar{x})$, where for all $i\in[d]$, $\bar{x}^{(i)}_k=x^{(i)}_k$, if $k\le k_i$ and $\bar{x}^{(i)}_k=x^{(i)}_{k_i}$, if $k\ge k_i$. We will also say that ${\rm n}$ is 
the smallest solution of the
system  $({\rm r},x)$.
\item[$(ii)$] We denote by $\mathscr{F}_d^{\rm r}$, the subset of $\mathscr{F}_d$ of $d$-type forests containing exactly  
$r_i$ roots of type $i$, for all $i\in[d]$.
\end{itemize}
\end{definition}

\noindent The following theorem gives a one to one correspondence between the sets $\mathscr{F}_d^{\rm r}$ and $\Sigma_d^{\rm r}$.

\begin{thm}\label{5678} Let  ${\rm r}=(r_1,\dots,r_d)\in\mathbb{Z}_+^d$, be such that  ${\rm r}>0$. Then for all 
$({\bf f},c_{\bf f})\in\mathscr{F}_d^{\rm r}$, the chain $x=\Psi({\bf f},c_{\bf f})$ belongs to the set $\Sigma_d^{\rm r}$. 
Moreover, the mapping
\begin{eqnarray*}
\Psi:\mathscr{F}_d^{\rm r}&\rightarrow&\Sigma_d^{\rm r}\\
({\bf f},c_{\bf f})&\mapsto&\Psi({\bf f},c_{\bf f})
\end{eqnarray*}
is a bijection.
\end{thm}
\begin{proof} In this proof, in order to simply the notation, we will identify the sequence $x$ with its extension $\bar{x}$ introduced 
in Definition \ref{9076}.

Let $({\bf f},c_{\bf f})$ be a forest of $\mathscr{F}_d^{\rm r}$ and let ${\rm s}=(s_1,\dots,s_d)\in\overline{\mathbb{Z}}_+^d$, where $s_i$ is the number of vertices of type $i$ in ${\bf f}$. We define a subtree of type $i\in[d]$ of $({\bf f},c_{\bf f})$ as a maximal connected subgraph of $({\bf f},c_{\bf f})$ whose all vertices are of type $i$. Formally, ${\bf t}$ is a subtree of type $i$ of $({\bf f},c_{\bf f})$, if it is a connected subgraph whose all vertices are of type $i$ and such that either $r({\bf t})$ has no parent or the type of its parent is different from $i$. Moreover, if the parent of a
vertex $v\in {\bf v}({\bf t})^c$ belongs to ${\bf v}({\bf t})$, then $c_{\bf f}(v)\neq i$. 

Let $i\in[d]$ and assume first that $s_i<\infty$ (i.e. $i\in U_{\rm s}$) and let $k_i\le s_i$ be the
number of subtrees of type $i$ in ${\bf f}$. Then we can check that the length $s_i$ of the sequence
$x^{i,i}$ corresponds to its first hitting time of level $-k_i$, i.e.
\begin{equation}\label{4191}
s_i=\inf\{n:x^{i,i}_n=-k_i\}\,.
\end{equation}
Indeed, let us rank the subtrees of type $i$ in ${\bf f}$ according to the breadth first search of their roots, 
so that we obtain the subforest of type $i$: 
${\bf t}_1,\dots,{\bf t}_{k_i}$ and let $x'$ be the Lukasiewicz-Harris coding path of this subforest (see, \cite{le}
or \cite{cl} for a definition of the Lukasiewicz-Harris coding path). Then we readily check that both sequences
end up at the same level, i.e.
\[\inf\{n:x'_n=-k_i\}=\inf\{n:x^{i,i}_n=-k_i\}\,.\]
Note that if $s_i=\infty$, then relation (\ref{4191}), whether or not $k_i$ is finite. 

Now let us check that ${\rm s}$ is a solution of the system $({\rm r},x)$,  that is 
\begin{equation}\label{5153}
r_j+\sum_{i=1}^dx^{i,j}(s_i)=0\,,\;\;\;j\in U_{\rm s}\,.
\end{equation} 
Let $j\in U_{\rm s}$, then $r_j+\sum_{i\neq j}x^{i,j}(s_i)$ clearly represents the total number of vertices of type
$j$ in ${\bf v}({\bf f})$ which are either a root of type $j$ or whose parent is of a type different from $j$, i.e. it represents the total number of subtrees of type $j$ in ${\bf f}$.  On the other hand, from (\ref{4191}), $-x^{j,j}(s_j)\ge0$ is the number of these subtrees. We conclude that equation (\ref{5153}) is satisfied.

It remains to check that ${\rm s}$ is the smallest solution of the system $({\rm r},x)$. As in  Lemma \ref{1377}, set ${\rm k}^{(0)}=0$ and for all $j\in[d]$,
\begin{equation}\label{7337}
k_j^{(n)}=\inf\big\{k:x^{j,j}_k=-(r_j+\sum_{i\neq j}x^{i,j}(k^{(n-1)}_i))\big\}\,,\;\;n\ge1\,.
\end{equation}
Then from the proof of  Lemma \ref{1377}, we have to prove that 
${\rm s}=\lim_{\rightarrow\infty}{\rm k}^{(n)}$.
Recall the coding which is defined in (\ref{6245}). For all $j\in[d]$, once we have visited the $r_j$ first vertices of type $j$ which are actually roots of the forest, we have to visit the whole corresponding subtrees, so that, if the total number of vertices of type $j$ in $({\bf f},c_{\bf f})$ is finite, i.e. $j\in U_{\rm s}$, then the chain $x^{j,j}$ first hits $-r_j$ at time $k_j^{(1)}<\infty$.  Then at time $k_j^{(1)}$, an amount of $\sum_{i\neq j}x^{i,j}(k^{(1)}_i)$ more subtrees of type $j$ have to be visited. So the chain $x^{j,j}$ has to hit $-(r_j+\sum_{i\neq j}x^{i,j}(k^{(1)}_i))$ at time $k^{(2)}<\infty$. This procedure is iterated until the last vertex of type $j$ is visited, that is until time $s_j=\lim_{n\rightarrow\infty}k_j^{(n)}<\infty$ (note that the sequence $k_j^{(n)}$ is constant after some finite index). Besides, from (\ref{7337}), we have
 \[s_j=\inf\{k:x^{j,j}_k=-(r_j+\sum_{i\neq j}x^{i,j}(s_i))\}\,,\;\;n\ge1\,,\;\;j\in
U_{\rm s}\,.\]
On the other hand, if the total number of vertices of type $j$ in $({\bf f},c_{\bf f})$ is infinite, then $k^{(n)}_j$ tends to $\infty$ 
(it can be infinite by some rank). So that we also have $s_j=\lim_{n\rightarrow\infty} k_j^{(n)}$ in this case. Therefore, ${\rm s}$ 
is the smallest solution of $({\rm r},x)$. 

Conversely let $x\in \Sigma_d^{\rm r}$ with length ${\rm s}$, then we construct a forest  
$({\bf f},c_{\bf f})\in\mathscr{F}_d^{\rm r}$, generation by generation, as follows. At generation $n=1$, 
for each $i\in[d]$, we take $r_i$ vertices of type $i$. We rank these $r_1+\dots+r_d$ vertices as it is defined 
in Subsection \ref{space}. Then to the $k$-th vertex of type $i$, we give $\Delta x_k^{i,j}:=x_k^{i,j}-x_{k-1}^{i,j}$ children of type $j\in[d]$ if $j\neq i$ and $\Delta x_k^{i,i}+1$ children of type $i$. We rank 
vertices of generation $n=2$ and to the $r_i+k$-th vertex of type $i$,  
we give $\Delta x_{r_i+k}^{i,j}$ children of type $j\in[d]$, if $j\neq i$ and $\Delta x_{r_i+k}^{i,i}+1$ children 
of type $i$, and so on. Proceeding this way, until the steps $\Delta x_{s_i}^{i,j}$, $i,j\in[d]$,
we have constructed a forest. Indeed the
total number of children of type $j$ whose parent is of type $i\neq j$ is $x^{i,j}(s_i)$, hence, the 
total number of children of type $j$ which is a root or whose parent is different from $j$ is 
$r_j+\sum_{i\neq j}x^{i,j}(s_i)$. 
Moreover,  each branch necessarily ends up with a leaf, since $\Delta x_{s_i+1}^{i,j}=0$, for all $i\neq j$ 
and $\Delta x_{s_i+1}^{i,i}=-1$.  Therefore we have constructed forest of $\mathscr{F}_d^{\rm r}$.

Finally, the  mapping which is described in $(\ref{6245})$ is clearly invertible, so we have proved that $\Psi$ is a bijection from $\mathscr{F}_d^{\rm r}$ to $\Sigma_d^{\rm r}$. 
$\;\;\Box$
\end{proof}

\subsection{Representing the sequence of generation sizes}\label{6737}
To each ${\bf f}\in \mathscr{F}_d^{\rm r}$ we associate the chain $z=(z^{(1)},\dots,z^{(d)})$ indexed by the
successive generations in ${\bf f}$, where for each $i\in[d]$, and $n\ge1$, $z^{(i)}(n)$ is the size of the population of type 
$i$ at generation $n$. More formally, we say that the (index of the) generation of $u\in{\bf v}({\bf f})$ is $n$ if 
$d(r({\bf t}_u),u)=n$, where ${\bf t}_u$ is the tree of ${\bf f}$ which contains $u$, $r({\bf t}_u)$ is the root of this tree and $d$ 
is the usual distance in discrete trees. Let us denote by $h({\bf f})$ the index of the highest generation in ${\bf f}$. 
Then $z^{(i)}$ is defined by 
\begin{equation}\label{8477}
z^{(i)}(n)=\left\{\begin{array}{ll}
\mbox{Card}\{u\in{\bf v}({\bf f}):c_{\bf f}(u)=i,\,d(r({\bf t}_u),u)=n\}&\mbox{if $n\le h({\bf f})$,}\\
0&\mbox{if $n\ge h({\bf f})+1$.}\\
\end{array}
\right.
\end{equation}
Let ${\bf v}^*({\bf f})$ be the subset of ${\bf v}({\bf f})$ of vertices which are not roots of ${\bf f}$. We denote by $u^*$ the parent 
of any $u\in{\bf v}^*({\bf f})$. We also define the chains $z^{i,j}$, for 
$i,j\in[d]$, as follows: $z^{i,j}(0)=0$ if $i\neq j$, $z^{i,i}(0)=z_i(0)=r_i$, for all $i,j\in[d]$, $z^{i,j}(n)=z^{i,j}(h({\bf f}))$ if $n\ge 
h({\bf f})$,  and for $n\ge1$,
\begin{equation}\label{2472}
z^{i,j}(n)=\left\{\begin{array}{l}
\mbox{Card}\{u\in{\bf v}^*({\bf f}):c_{\bf f}(u)=j,\,c_{\bf f}(u^*)=i,\,d(r({\bf t}_u),u)\le n\}\,,\;\;\mbox{if $i\neq j$,}\\
r_i+\mbox{Card}\{u\in{\bf v}^*({\bf f}):c_{\bf f}(u)=i,\,c_{\bf f}(u^*)=i,\,p_i(u^*)\ge2,\,d(r({\bf t}_u),u)\le n\}-\\
\mbox{Card}\{u\in{\bf v}({\bf f}):c_{\bf f}(u)=i,\,p_i(u)=0,\,d(r({\bf t}_u),u)\le n-1\}\,,\;\;\mbox{if $i= j$.}
\end{array}\right.
\end{equation}
In words, if $i\neq j$ then $z^{i,j}(n)$ is the total number of vertices of type $j$ whose parent is of type $i$ in the $n$
first generations of the forest ${\bf f}$. If $i=j$ then we only count the number of vertices of type $i$ with at least one 
brother of type $i$ and whose parent is of type $i$ in the $n$ first generations. To this number, we subtract the 
number of vertices of type $i$  with no children of type $i$, whose generation is less  or equal than $n-1$. Then it is 
not difficult to check the  following relation:
\begin{equation}\label{2671}
z^{(i)}(n)=\sum_{i=1}^dz^{i,j}(n)\,,n\ge0\,,\;\;\;i\in[d]\,.
\end{equation}

\noindent We end this subsection by a lemma which provides a relationship between the chains $x$ and $z$ and $z^{i,j}$. 
This result is the deterministic expression of the Lamperti representation of Theorem  \ref{5891} below and its continuous 
time counterpart in Theorems \ref{2379} and \ref{7462}.
\begin{lemma}\label{6209} The chain $z^{i,j}$ may be obtained as the following time change of the chain $x^{i,j}$:
\begin{eqnarray}
z^{i,j}(n)&=&x^{i,j}(\mbox{$\sum_{k=0}^{n-1}z^{(i)}(k)$})\,,\;\;\;n\ge1\,,\;\;\;i,j\in[d]\,,\;\;i\neq j\,,\label{8377}\\
z^{i,i}(n)&=&r_i+x^{i,i}(\mbox{$\sum_{k=0}^{n-1}z^{(i)}(k)$})\,,\;\;\;n\ge1\,,\;\;\;i\in[d]\,.\label{8517}
\end{eqnarray}
In particular, we have 
\begin{equation}\label{9537}
z^{(j)}(n)=r_j+\sum_{i=1}^dx^{i,j}(\mbox{$\sum_{k=0}^{n-1}z^{(i)}(k)$})\,,\;\;\;n\ge1\,,\;\;\;j\in[d]\,.
\end{equation}
Moreover, given $x^{i,j}$, $i,j\in[d]$, the chains $z^{i,j}$, $i,j\in[d]$ are 
uniquely determined by equations $(\ref{2671})$, $(\ref{8377})$ and
$(\ref{8517})$. 
\end{lemma}
\begin{proof}
It suffices to check relations (\ref{8377}) and (\ref{8517}). Then (\ref{9537}) will follow from (\ref{2671}). 
But (\ref{8377}) and (\ref{8517}) are direct consequences of the definition of the chains $x^{i,j}$ and $z^{i,j}$ 
in (\ref{6245}) and (\ref{2472}) respectively. $\;\;\Box$
\end{proof}

\subsection{Application to discrete time branching processes}\label{6477}
Recall that $\nu=(\nu_1,\dots,\nu_d)$ is a progeny distribution, such that the $\nu_i$'s are any probability
measures on $\mathbb{Z}_+^d$, such that $\nu_i(e_i)<1$. 
Let $(\Omega,\mathcal{G},P)$ be some measurable space on which, for any ${\rm r}\in\mathbb{Z}_+^d$ such that 
${\rm r}>0$, we  can define a probability measure $\p_{\rm r}$ and a  random variable 
$(F,c_F):(\Omega,\mathcal{G},\p_{\rm r})\rightarrow\mathscr{F}_d^{\rm r}$ whose law under $\p_{\rm r}$ is this of a branching 
forest with progeny law $\nu$. Then we construct from $(F,c_F)$ the random chains, 
$X=(X^{(1)},\dots,X^{(d)})$, $Z=(Z^{(1)},\dots,Z^{(d)})$ and $Z^{i,j}$, $i,j\in[d]$, exactly as in (\ref{6245}), (\ref{8477}) and 
(\ref{2472}) respectively. In particular, $X=\Psi(F,c_F)$. We can check that under $\p_{\rm r}$, $Z$ is a branching 
process with progeny distribution $\nu$. More specifically, recall from (\ref{4590}) the definition of
$\tilde{\nu}_i$, then the random processes $X$ and $Z$ satisfy the following result. 

\begin{thm}\label{5891}
Let ${\rm r}\in\mathbb{Z}_+^d$ be such that ${\rm r}>0$ and let $(F,c_F)$ be an 
$\mathscr{F}_d^{\rm r}$-valued branching forest with progeny law $\nu$ under $\p_{\rm r}$. Let  $N=(N_1,\dots,N_d)\in\overline{\mathbb{Z}}_+^d$,
where $N_i$ is the number of vertices of type $i$ in $F$.Then,
\begin{itemize}
\item[$1.$]  The random variable $N$ is almost surely the smallest solution of the system $({\rm r},X)$ in the sense of Definition $\ref{9076}$. If $\nu$ is irreducible, non degenerate and $($sub$)$critical, then almost
surely, $N_i<\infty$, for all $i\in[d]$. If $\nu$ is irreducible, non degenerate and super critical, then with some
probability $p>0$, $N_i=\infty$, for all $i\in[d]$ and with probability $1-p$,  $N_i<\infty$, for all $i\in[d]$. 
\item[$2.$]   On the space $(\Omega,\mathcal{G},P)$, we can define independent random walks $\tilde{X}^{(i)}$, $i\in[d]$, with respective step distributions $\tilde{\nu}_i$, $i\in[d]$, such that  $\tilde{X}_0^{(i)}=0$ and if $\tilde{N}=(\tilde{N}_1,\dots\tilde{N}_d)\in\overline{\mathbb{Z}}_+^d$
is the smallest solution of the system $({\rm r},\tilde{X})$, then the following identity in law
\[(X^{(i)}_k,\,0\le k\le N_i,\,i\in [d])\ed(\tilde{X}^{(i)}_k,\,0\le k\le \tilde{N}_i,\,i\in [d])\]
holds. 
\item[$3.$] The joint law of $X$ and $N$ is given as 
follows: for any integers $n_i$ and
$k_{ij}$, $i,j\in [d]$, such that $n_i>0$, $k_{ij}\in\mathbb{Z}_+$, for $i\neq j$,
  $-k_{jj}=r_j+\sum_{i\neq j}k_{ij}$ and $n_i\ge-k_{ii}$,
\begin{eqnarray*}
&&\p_{\rm r}\Big(X^{i,j}_{n_i}=k_{ij},\,\mbox{$ i,j\in [d]$ and $N={\rm n}$}\Big)=\\
&&\qquad\qquad\frac{\mbox{\rm det}(-k_{ij})}{n_1n_2\dots n_d}\prod_{i=1}^d\nu_i^{*n_i}\Big(k_{i1},\dots,k_{i(i-1)},n_i+k_{ii},k_{i(i+1)},\dots,k_{id}\Big)\,.\label{5478}
\end{eqnarray*}
\item[$4.$] The random process $Z$ is a branching process with progeny law $\nu$, which is related to $X$ through the 
time change:
\begin{equation}\label{7371}
Z^{(i)}(n)=\sum_{i=1}^dX^{i,j}(\mbox{$\sum_{k=1}^{n-1}Z^{(i)}(k)$})\,,\;\;\;n\ge1\,.
\end{equation}
\end{itemize}
\end{thm}

\begin{proof} The fact that $N$ is the smallest solution of the system $({\rm r},X)$ is a direct consequence of Theorem \ref{5678} 
and the definition of $X$, that is $\Psi(F,c_F)=X$. Assume that $\nu$ is irreducible, non
degenerate and (sub)critical. Then since the forest $F$ contains almost surely a finite number of vertices, all coordinates of $N$ 
must be finite from Theorem \ref{5678}. If $\nu$ is irreducible, non
degenerate and supercritical, then with probability $p>0$ the forest $F$ contains an infinite number of
vertices of type $i$, for all $i\in[d]$  and with probability $1-p$ its total population is finite. Then 
the result also follows from Theorem \ref{5678}.

In order to prove part 2, let $(F_n,c_{F_n})$ with  $(F_1,c_{F_1})=(F,c_F)$, be a sequence of independent
and identically distributed forests. Let us define $X^n=\Psi(F_n,c_{F_n})$ and then let us concatenate the
processes $X^n=(X^{n,(1)},\dots,X^{n,(d)})$ in a process that we denote $\tilde{X}$. More specifically, let
us denote by $N_i^n$ the length of $X^{n,(i)}$, then the process obtained from this concatenation is 
$\tilde{X}=(\tilde{X}^{(1)},\dots,\tilde{X}^{(d)})$, where $\tilde{X}^{(i)}_0=0$,  $N_i^0=0$ and
\begin{eqnarray*}
\tilde{X}_k^{(i)}&=&\tilde{X}_{N_i^0+\dots+N_i^{n-1}}^{(i)}+X^{n,(i)}_{k-(N^0_i+\dots+N_i^{n-1})}\,,\\
&&\qquad\qquad\mbox{if $N_i^0+\dots+N_i^{n-1}\le  k\le N_i^0+\dots+N_i^{n}$}\,,\;\;n\ge1\,.
\end{eqnarray*}
Note that $\tilde{X}$ is obtained by coding the forests  $(F_n,c_{F_n})$, $n\ge1$ successively. 
Then it readily follows from the construction of $\tilde{X}$ and the branching property that the coordinates
$\tilde{X}^{(i)}$ are independent random walks with step distribution $\tilde{\nu}_i$. Moreover $N$ is a 
solution of the system $({\rm r},\tilde{X})$, so its smallest solution, say $N'$, is necessarily smaller than $N$.
This means that $N'$ is a solution of the system $({\rm r},X)$, hence $N'=N$.

The third part is a direct consequence of the first part and the multivariate ballot Theorem which is proved in \cite{cl}, see Theorem 3.4 therein.

Then part 4, directly follows from the definition of $Z$ and Lemma \ref{6209}.
$\;\;\Box$\\
\end{proof}

Conversely, from any random walk whose step distribution is given by (\ref{4590}) we can construct  a unique branching forest 
with law $\p_{\rm r}$. The following result is a direct consequence of Theorems \ref{5891}  and \ref{5678}.

\begin{thm}\label{1259} Let $\tilde{X}^{(i)}$, $i\in[d]$ be $d$ independent random walks defined on $(\Omega,\mathcal{G},P)$, whose respective step distributions are $\tilde{\nu}_i$, and set $\tilde{X}=(\tilde{X}^{(1)},\dots,\tilde{X}^{(d)})$. Let  ${\rm r}\in\mathbb{Z}_+^d$ such that ${\rm r}>0$ and let $\tilde{N}$ be the smallest solution of the system $({\rm r},\tilde{X})$. We 
define the $\Sigma^{\rm r}_d$-valued process $X=(X^{(1)},\dots,X^{(d)})$ by
$(X^{(i)}_k,\,0\le k\le \tilde{N}_i)=(\tilde{X}^{(i)}_k,\,0\le k\le \tilde{N}_i)$.

Then $(F,c_F):=\Psi^{-1}(X)$ is a $\mathscr{F}_d^{\rm r}$-valued branching forest  $(F,c_F)$ with progeny distribution $\nu$. 
\end{thm}

\section{The continuous time setting}\label{proofs}

\subsection{Proofs of Theorem  \ref{6803} and Proposition \ref{1269}.}\label{1842} Let $Y=(Y^{(1)},\dots,Y^{(d)})$ be 
the underlying random walk of the compound Poisson process $X$, that is the random walk such that there
are independent Poisson processes $N^{(i)}$, with parameters $\lambda_i$ such that $X^{(i)}_t=Y^{(i)}(N^{(i)}_t)$ and $(N^{(i)},Y^{(j)},\,i,j\in[d])$ are independent. Then from Lemma \ref{1377}, there is a random time ${\rm s}\in\overline{\mathbb{Z}}_+^d$, such that almost surely, for all 
$j\in U_{\rm s}$, $x_j+\sum_{i=1}^dY^{i,j}(s_i)=0$. Moreover, if ${\rm s}'$ is any time satisfying this property, then ${\rm s}'\ge{\rm s}$.  For $i\in[d]\setminus U_{\rm s}$, the latter equality implies that $Y^{i,j}(\infty)<\infty$. Since $Y^{i,j}$ is a renewal process, it is possible only 
if $Y^{i,j}\equiv0$, a.s., so that we can write:  almost surely, 
\[x_j+\sum_{i,\,i\in U_{\rm s}}Y^{i,j}(s_i)=0\,,\;\;\;\mbox{for all $j\in U_{\rm s}$}\,.\]
Then the first part of the Theorem follows from the construction of $X$ as a time change of $Y$. More formally, the coordinates of $T_x$ are given by $T^{(i)}_x=\inf\{t:N^{(i)}(t)={\rm s}_i\}$, so that in particular,
$s_i=N^{(i)}(T^{(i)}_x)$.

Additivity property of $T_x$ is a consequence of Lemma \ref{1377} and time change. From this lemma, we can deduct that for all $x,y\in\mathbb{Z}_+^d$, if ${\rm s}$ is the smallest solution of $(x+y,Y)$, then conditionally to $s_i<\infty$, for all $i\in[d]$, the smallest solution ${\rm s}_1=(s_{1,1},\dots,s_{1,d})$ of $(x,Y)$, and satisfies ${\rm s}_1\le {\rm s}$ and 
$s-s_1$ is the smallest solution of the system $(y,\tilde{Y})$, where 
$\tilde{Y}^{(i)}_k=Y^{(i)}_{s_{1,i}+k}-Y^{(i)}_{s_{1,i}}$, $k\ge0$. Moreover, $\tilde{Y}=(\tilde{Y}^{(i)},i\in[d])$ 
has the same law as $Y$ and is independent of $(Y_k^{(i)},\,0\le k\le s_{1,i})$. Using the time change, we obtain,
\[L_{x+y}\ed L_x+\tilde{L}_y\,,\]
where $L_x$ has the law of $T_x$ conditionally on $T_{x}^{(i)}<\infty$, for all $i\in[d]$ and $\tilde{L}_y$
is an independent copy of $L_y$. Then identity (\ref{6290}) follows.

The law of $T_x$ on $\mathbb{R}_+^d$ follows from time change and the same result in the discrete time
case obtained in \cite{cl}, see Theorem 3.4 therein.

Proposition \ref{1269} is a direct consequence part 1 of Theorem \ref{6803} and the time change.

\subsection{Multitype forests with edge lengths}\label{2312} 
A $d$ type forest with edge lengths is an element
$({\bf f},c_{\bf f},\ell_{\bf f})$, where  $({\bf f},c_{\bf f})\in\mathscr{F}_d$ and $\ell_{\bf f}$ is some application 
$\ell_{\bf f}:{\bf v}({\bf f})\rightarrow(0,\infty)$. For $u\in{\bf v}({\bf f})$, the quantity $\ell_{\bf f}(u)$ will be called the life time 
of $u$. It corresponds to the length of an edge incident to $u$ in $({\bf f},c_{\bf f},\ell_{\bf f})$ whose color is this of $u$.
This edge is a segment which is closed at the extremity corresponding to $u$ and open at the other extremity. 
If $u$ is not a leaf of $({\bf f},c_{\bf f})$ then $\ell_{\bf f}(u)$ corresponds to the length of the edge between $u$ and its 
children in the continuous forest $({\bf f},c_{\bf f},\ell_{\bf f})$. To each tree of $({\bf f},c_{\bf f})$ corresponds a tree of  
$({\bf f},c_{\bf f},\ell_{\bf f})$ which is considered as a continuous metric space, the distance being given by the Lebesgue 
measure along the branches. To each forest $({\bf f},c_{\bf f},\ell_{\bf f})$ we associate a time scale such that a vertex
$u$ is born at time $t$ if the distance between $u$ and the root of its tree in $({\bf f},c_{\bf f},\ell_{\bf f})$ is
$t$. Time $t$ is called the birth time of $u$ in  $({\bf f},c_{\bf f},\ell_{\bf f})$ and it is denoted by $b_{\bf f}(u)$.
The death time of $u$ in 
$({\bf f},c_{\bf f},\ell_{\bf f})$ is then $b_{\bf f}(u)+\ell_{\bf f}(u)$. If $s\in [b_{\bf f}(u),b_{\bf f}(u)+\ell_{\bf f}(u))$ then we 
say that $u$ is alive at time $s$ in the forest  $({\bf f},c_{\bf f},\ell_{\bf f})$. We denote by $\mbox{h}_{\bf f}$ the smallest 
time when no individual is alive in $({\bf f},c_{\bf f},\ell_{\bf f})$.

The set of $d$ type forests with edge lengths will be denoted by $F_d$. The subset of $F_d$ of 
elements $({\bf f},c_{\bf f},\ell_{\bf f})$ such that $({\bf f},c_{\bf f})\in\mathscr{F}_d^{\rm r}$ will be denoted 
by $F_d^{\rm r}$. Elements of $F_d$ will be represented as in Figure \ref{fig1}.\\

To each forest $({\bf f},c_{\bf f},\ell_{\bf f})\in F_d^{\rm r}$, we associate the multidimensional the step
functions, $(z^{(i)}(t),t\ge0)$ that are defined as follows:

\begin{equation}\label{2037}
z^{(i)}(t)=\mbox{Card}\{u\in{\bf v}({\bf f}):c_{\bf f}(u)=i,\,\mbox{$u$ is alive at time $t$ in $({\bf f},c_{\bf f},\ell_{\bf f})$.}\}
\end{equation}
Then $(z^{i,j}(t),t\ge0)$  is defined by

\begin{equation}\label{2847}
z^{i,j}(t)=\left\{\begin{array}{l}
\mbox{Card}\{u\in{\bf v}^*({\bf f}):c_{\bf f}(u)=j,\,c_{\bf f}(u^*)=i,\,b_{\bf f}(u)\le t\}\,,\;\;\mbox{if $i\neq j$,}\\
r_i+\mbox{Card}\{u\in{\bf v}^*({\bf f}):c_{\bf f}(u)=i,\,c_{\bf f}(u^*)=i,\,b_{\bf f}(u)\le t\}-\\
\mbox{Card}\{u\in{\bf v}({\bf f}):c_{\bf f}(u)=i,\,b_{\bf f}(u)+\ell_{\bf f}(u)\le t\}\,,\;\;\mbox{if $i= j$.}
\end{array}\right.
\end{equation}
It readily follows from these definitions that 
\[z^{(j)}(t)=\sum_{i=1}^dz^{i,j}_t\,,\;\;\;t\ge0\,.\]

We now define the discretisation of forests of $F_d$, with some span $\delta>0$. Let $({\bf f},c_{\bf f},\ell_{\bf f})\in F_d$, 
then on each tree of 
$({\bf f},c_{\bf f},\ell_{\bf f})\in F_d$, we place  new vertices at distance $n\delta$, $n\in\mathbb{Z}_+$ of the root along 
all the branches. A vertex which is placed along an edge with color $i$ has also color $i$. Then we define a forest in
$\mathscr{F}_d$ as follows. A new vertex $v$ is the child of a new vertex $u$ if and only if both vertices belong to the 
same branch of $({\bf f},c_{\bf f},\ell_{\bf f})$, and there is $n\ge0$ such that $u$ and $v$ are respectively at distance $n\delta$ 
and $(n+1)\delta$ from the root. This transformation defines an application which we will denote by
\begin{eqnarray*}
D_\delta:F_d&\rightarrow&\mathscr{F}_d\\
({\bf f},c_{\bf f},\ell_{\bf f})&\mapsto&D_\delta({\bf f},c_{\bf f},\ell_{\bf f})\,.
\end{eqnarray*}
Note that with this definition, the roots of the three forests $({\bf f},c_{\bf f})$, $({\bf f},c_{\bf f},\ell_{\bf f})$ and 
$D_\delta({\bf f},c_{\bf f},\ell_{\bf f})$ are the same and more generally, a vertex of $D_\delta({\bf f},c_{\bf f},\ell_{\bf f})$
corresponds to a vertex $u$ of $({\bf f},c_{\bf f},\ell_{\bf f})$ if and only if $u$ is at a distance equal to $n\delta$ from the 
root. The definition of the discretisation of a forest with edge lengths should be obvious from Figure \ref{fig3}.

\begin{figure}[hbtp]

\vspace*{.2cm}

\hspace*{-0.5cm}{\includegraphics[height=350pt,width=500pt]{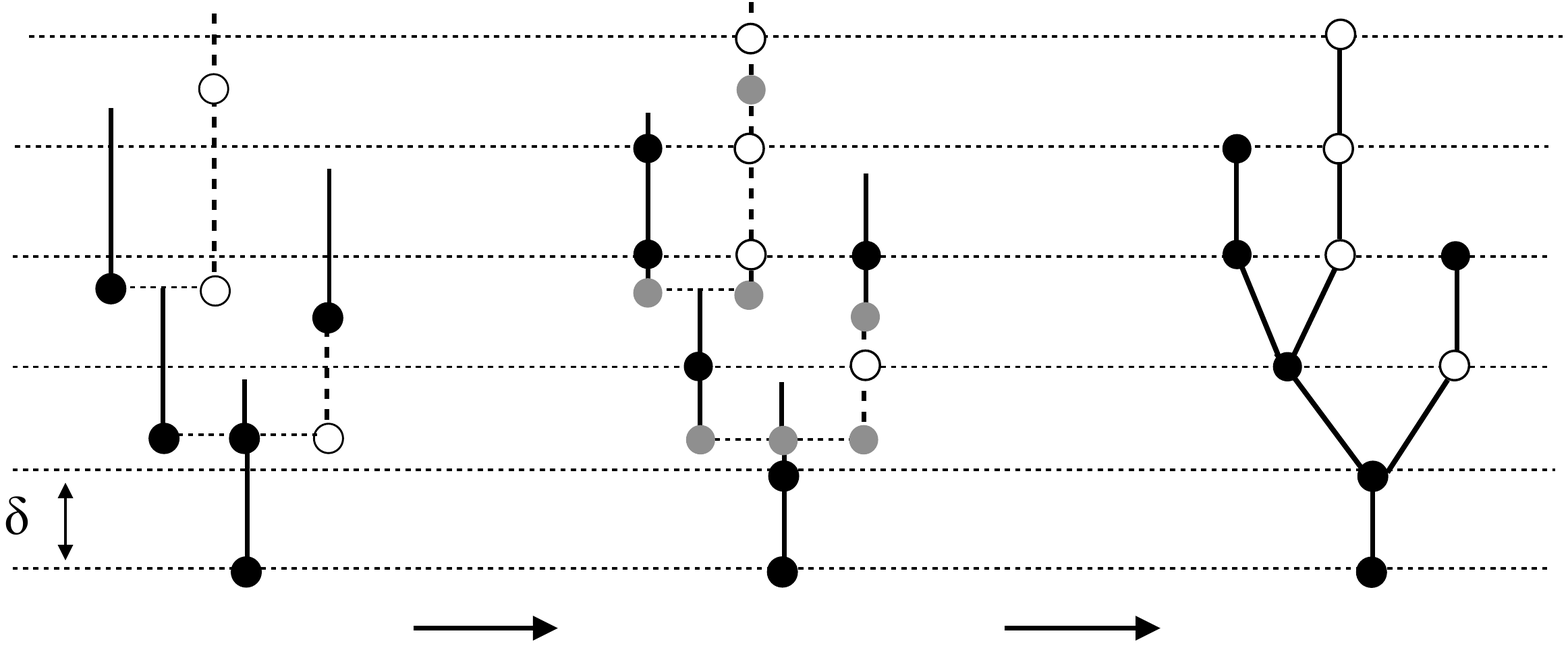}}
 
\vspace{-6cm}

\caption{Discretisation of a two type tree with edge lengths.}\label{fig3}
\end{figure}

\subsection{Multitype branching forests with edge lengths.}  Recall from Section \ref{main} that 
$\lambda_1,\dots,\lambda_d$ are positive, finite and constant rates, and that $\nu=(\nu_1,\dots,\nu_d)$, where $\nu_i$,
$i\in[d]$ are any distributions in $\mathbb{Z}_+^d$. From the setting established in Subsection \ref{2312}, we can define 
a branching forest with edge lengths as a random variable $(F,c_F,\ell_F):(\Omega,\mathcal{G},\p_{\rm r})\rightarrow 
(F_d,\mathcal{H}_d)$, where $\mathcal{H}_d$ is the sigma field of the Borel sets of $F_d$ endowed with the Gromov-Hausdorff topology (see Section 2.1 in \cite{le}) and where the law of   
$(F,c_F):(\Omega,\mathcal{G},\p_{\rm r})\rightarrow\mathscr{F}_d^{\rm r}$ under  $\p_{\rm r}$
is this of a discrete branching forest with progeny distribution $\nu$, as defined in Subsection \ref{6477}. Besides, let $N_i$
 be the number of vertices of type $i$ in $(F,c_F)$, then for all ${\rm n}=(n_1,\dots,n_d)\in\overline{\mathbb{Z}}_+^d$, conditionally  on $N_i=n_i$, $i\in[d]$, $(\ell(u_n^{(i)}))_{0\le n\le n_i}$ are sequences of i.i.d. exponentially distributed random
variables with respective parameters $\lambda_i$, and $[(\ell(u_n^{(i)}))_{0\le n\le n_i},\,i\in[d],\,(F,c_F)]$ are independent.\\

\noindent Then we have the following result which is straightforward from the above definitions.

\begin{proposition}\label{7242}
Let $(F,c_F,\ell_F)$ be  a branching forest with edge lengths with progeny distribution  $\nu=(\nu_1,\dots,\nu_d)$ on
$\mathbb{Z}_+^d$ and life time rates $\lambda_1,\dots,\lambda_d\in(0,\infty)$. Then for all ${\rm r}$, under $\p_{\rm r}$, 
the process $Z=(Z_i(t),\,i\in[d]\,,t\ge0)$ which is defined as in $(\ref{2037})$ with respect to $(F,c_F,\ell_F)$ is a 
continuous time, $\mathbb{Z}_+^d$-valued branching process starting at $Z_0={\rm r}$, with progeny distribution
$\nu=(\nu_1,\dots,\nu_d)$ and life time rates $\lambda_1,\dots,\lambda_d$.
\end{proposition}

\noindent The law of a discretised branching forest with edge lenghts is given by the following lemma. 

\begin{lemma}\label{4166} Let $(F,c_F,\ell_F)$ be  a branching forest with edge lengths, with progeny distribution  
$\nu=(\nu_1,\dots,\nu_d)$ on
$\mathbb{Z}_+^d$ and life time rates $\lambda_1,\dots,\lambda_d\in(0,\infty)$.
For $\delta>0$, the forest $D_\delta(F,c_F,\ell_F)$ is a $($discrete$)$ branching forest with progeny distribution:
\[\nu_i^{(\delta)}({\rm k})=\p_{e_i}(Z_\delta={\rm k})\,,\;\;\;{\rm k}=(k_1,\dots,k_d)\in\mathbb{Z}_+^d\,,\]
where $e_i$ is the $i$-th unit vector of $\mathbb{Z}_+^d$.
\end{lemma}
\begin{proof} The fact that $D_\delta(F,c_F,\ell_F)$ is a discrete branching forest is a direct consequence of its construction.
Indeed, at generation $n$, that is at time $n\delta$, the vertices of this forest inherits from the branching property of 
$(F,c_F,\ell_F)$ the fact they will give birth to some progeny, independently of each other and with some distribution which only depends on their type and $\delta$. Then its remains  to determine the progeny distribution $\nu^{(\delta)}$. But it is
obvious from the construction of $D_\delta(F,c_F,\ell_F)$ that $\nu_i^{(\delta)}$ is the law of the total offspring at time 
$\delta$ of a root of type $i$ in the forest $(F,c_F,\ell_F)$. This is precisely the expression which is given in the 
statement.  $\;\;\Box$
\end{proof}

\subsection{Proof of Theorem  \ref{2379}.}  

Let $(F,c_F,\ell_F)$ be  a branching forest issued from $x$, with edge lengths, with progeny distribution  $\nu=(\nu_1,\dots,\nu_d)$ 
on $\mathbb{Z}_+^d$ and life time rates $\lambda_1,\dots,\lambda_d\in(0,\infty)$. Then from Proposition \ref{7242}, the process 
$Z=(Z^{(i)}(t),\,i\in[d],\,t\ge0)$ which is defined as in $(\ref{2037})$ with respect to $(F,c_F,\ell_F)$ is a continuous time, 
$\mathbb{Z}_+$-valued branching process with progeny distribution $\nu=(\nu_1,\dots,\nu_d)$ and life time rates 
$\lambda_1,\dots,\lambda_d$.\\ 

Let $\delta>0$ and consider the discrete forest $D_\delta(F,c_F,\ell_F)$ whose progeny 
distribution is given by Lemma \ref{4166}. Let ${\rm Z}^{\delta}=({\rm Z}^{\delta,(1)},\dots,{\rm Z}^{\delta,(d)})$ be the associated (discrete time) branching process and let 
${\rm Z}^{\delta,i}:=({\rm Z}^{\delta,i,1},\dots,{\rm Z}^{\delta,i,d})$, $i\in[d]$ be the decomposition 
of ${\rm Z}^{\delta}$, as it is defined in (\ref{2472}). Then it is straightforward that 
\begin{equation}\label{5792}
({\rm Z}^{\delta,i,j}([\delta^{-1} t])\,,t\ge0)\rightarrow (Z^{i,j}_t\,,t\ge0)\,,
\end{equation}
almost surely on the Skohorod's space of c\`adl\`ag paths toward the process $Z^{i,j}$, as $\delta$ tends to 0, for all $i,j\in[d]$, 
where  $Z^{i,j}$ is the decomposition of $Z$ as it is defined in (\ref{2847}).\\

Now, let ${\rm X}^\delta=({\rm X}^{\delta,{(i)}},\,i\in[d])$ be the coding random walk associated to $D_\delta(F,c_F,\ell_F)$,
as in Theorem \ref{5891}. We will first assume that ${\rm X}^\delta$ is actually the coding random walk of a sequence of  i.i.d. discrete forests with the same distribution as $D_\delta(F,c_F,\ell_F)$ in the same manner
as in the proof of part 2 of Theorem \ref{5891}, so that in particular, ${\rm X}^\delta$ is not a stopped random walk.\\

For $i\in[d]$, let  $\tau^{{\rm X}^\delta}_{i,n}$ and $\tau^{{\rm Z}^\delta}_{i,n}$ be the sequences of 
jump times of the discrete time processes ${\rm X}^{\delta,(i)}$ and ${\rm Z}^{\delta,i}$. That is the ordered sequences of times 
such that $\tau^{{\rm X}^\delta}_{i,0}=\tau^{{\rm Z}^\delta}_{i,0}=0$ and
$\Delta {\rm X}^{\delta,(i)}_n:= {\rm X}^{\delta,(i)}(\tau^{{\rm X}^\delta}_{i,n})-{\rm X}^{\delta,(i)}(\tau^{{\rm X}^\delta}_{i,n-1})\neq0$ 
and  $\Delta {\rm Z}^{\delta,i}_n:={\rm Z}^{\delta,i}(\tau^{{\rm Z}^\delta}_{i,n})-{\rm Z}^{\delta,i}
 (\tau^{{\rm Z}^\delta}_{i,n-1})\neq0$, for $n\ge1$. Fix $k\ge1$, then since two jumps of the
process $Z$ almost surely never occur at the same time, there is $\delta_0$,
sufficiently small such that for all $0<\delta\le \delta_0$, the sequences  $(\Delta {\rm X}^{\delta,(i)}_n,\,0\le n\le k)$ and 
$(\Delta {\rm Z}^{\delta,i}_n,\,0\le n\le k)$ are a.s. identical, for all $i\in[d]$.  Moreover, from 
Lemma \ref{4166} and Theorem \ref{5891}, the jumps $\Delta {\rm X}^{\delta,(i)}_n$ have law 
\[\mu_i^{(\delta)}({\rm k})=\frac{\tilde{\nu}_i^{(\delta)}({\rm k})}{1-\tilde{\nu}_i^{(\delta)}(0)}\,,\;\;{\rm k}\neq 0\,,\;\;
\mu_i^{(\delta)}(0)=0\,,\]
where $\tilde{\nu}_i^{(\delta)}({\rm k})=\p_{e_i}(Z_\delta=(k_1,\dots,k_{i-1},k_i+1,k_{i+1},\dots,k_d))$. The measure 
$\mu^{(\delta)}_i$ converges weakly to $\mu_i$ defined in (\ref{7458}), as $\delta\rightarrow0$. Hence from (\ref{5792}), the 
sequence $(\Delta {\rm Z}^{\delta,i}_n,\,n\ge 0)$ converges almost surely toward the sequence $(\Delta Z^{i}_n,\,n\ge 0)$ of 
jumps of the process $(Z^i_t,\,t\ge0):=(Z^{i,j}_t,\,j\in[d]\,,t\ge0)$, which is therefore a sequence of i.i.d. random variables, with 
law $\mu_i$.\\

On the other hand, it follows from (\ref{8377}) and (\ref{8517}) in Lemma \ref{6209} and the fact that two 
jumps of the process $Z$ almost surely never occur at the same time, that for all $n_1$, there is
$\delta_1>0$ sufficiently small, such that for all $n\le n_1$ and $0<\delta\le\delta_1$,
\[\tau^{{\rm X}^\delta}_{i,n}-\tau^{{\rm X}^\delta}_{i,n-1}=
\sum_{k=\tau^{{\rm Z}^\delta}_{i,n-1}}^{\tau^{{\rm Z}^\delta}_{i,n}}{\rm Z}^{\delta,(i)}_k\,,\;\;\;n\ge1\,.\]
From Lemma \ref{4166} , the latter is a sequence of i.i.d. geometrically distributed random variables with parameter 
$1-\p_{e_i}(Z_\delta=e_i)$.  Hence from (\ref{5792}), the sequence 
\[\delta\cdot\sum_{k=\delta^{-1}\tau^{{\rm Z}^\delta}_{i,n-1}}^{\delta^{-1}\tau^{{\rm Z}^\delta}_{i,n}}{\rm Z}^{\delta,(i)}_k\,,
\;\;\;n\ge1\]
converges almost surely toward the sequence
\[\int_{\tau^{Z}_{i,n-1}}^{\tau^{Z}_{i,n}} Z^{(i)}_t\,dt\,,\;\;\;n\ge1\,,\]
as $\delta\rightarrow0$, where $(\tau^{Z}_{i,n})$ is the sequence of jump times of $Z^{i}$. 
Moreover the variables of this sequence are i.i.d. and exponentially distributed with parameter 
$\lim_{\delta\rightarrow0}\delta^{-1}(1-\p_{e_i}(Z_\delta=e_i))=\lambda_i$.\\ 

Then since $({\rm X}^{\delta,(i)}_n)$ is a random walk, 
the sequences $(\Delta {\rm X}^{\delta,(i)}_n)_{n\ge0}$, $(\tau^{{\rm X}^\delta}_{i,n})_{n\ge0}$, $i\in[d]$ are independent.
Therefore, from the convergences proved above, the sequences $(\Delta Z^{i}_n,\,n\ge 0)$ and 
$(\int_{\tau^{Z}_{i,n-1}}^{\tau^{Z}_{i,n}} Z^{(i)}_t\,dt\,,n\ge1)$ are independent. Then we have proved that 
the process 
\begin{equation}\label{1547}
X^{(i)}:=(Z^{i}(\tau_t^{(i)}),\,t\ge0)\,,\;\;\mbox{ where $\tau_t^{(i)}=\inf\{s:\int_0^s Z^{(i)}_u\,du>t\}$,}
\end{equation}
is a compound Poisson process  with parameter  $\lambda_i$ and jump distribution $\mu_i$. Moreover, it follows from 
the independence between the random walks  $({\rm X}^{\delta,(i)}_n)$, $i\in[d]$ that the processes $(Z^{i}(\tau_t^{(i)}),\,t\ge0)$, $i\in[d]$ are independent.\\

Now from part 1 of Theorem \ref{5891}, if $N^{\delta}_i$ is the total population of type $i\in[d]$ in the forest  
$D_\delta(F,c_F,\ell_F)$, then $N^{\delta}=(N^{\delta}_1,\dots,N^{\delta}_d)$ is the smallest solution of the system $(x,{\rm X}^{\delta})$. Moreover it follows from the construction of $D_\delta(F,c_F,\ell_F)$, that 
$\lim_{\delta\rightarrow0}\delta N^{\delta}_i=\int_0^\infty Z^{(i)}_t\,dt:=T_x^{(i)}$, almost surely. Note that
$T_x^{(i)}$ represents the total length of the branches of type $i$ in the forest $(F,c_F,\ell_F)$. Then from 
the definition of $X^{(i)}$ in (\ref{1547}) it appears the these compound Poisson processes are stopped at $T^{(i)}_x$  and that $T_x=(T_x^{(1)},\dots,T_x^{(d)})$ satisfies (\ref{3678}).\\

The fact that $(\ref{3521})$ is invertible is a direct consequence the first part of the following lemma.

\begin{lemma}\label{1592}Let $x=(x_1,\dots,x_d)\in\mathbb{Z}_+^d$ and $\{(x^{i,j}_t,\,t\ge0),\;i,j\in[d]\}$ be 
a family of c\`adl\`ag $\mathbb{Z}$-valued, step functions such that  for $i\neq j$, $x^{i,j}=0$, $x^{i,j}$ are
increasing, $x^{i,i}_0=x_i\ge0$ and $x^{i,i}$ are downward skip free, i.e. 
$x^{i,i}_t-x^{i,i}_{t-}\ge-1$, for all $t\ge0$, with $x^{i,i}_{0-}=x_i$. Then there exists a $($unique$)$ 
time $t_x=(t_x^{(1)}\dots,t_x^{(d)})\in\overline{\mathbb{R}}_+^d$ such that 
\begin{equation}\label{3138}
x_j+\sum_{i=1}^dx^{i,j}(t_x^{(i)})=0\,,\;\;\;\mbox{for all $j$ such that $t^{(j)}_x<\infty$}\,,
\end{equation}
and if $t'_x$ is any time satisfying $(\ref{3138})$, then  $t'_x\ge t_x$. 

Moreover, the  system
\[z^{i,j}_t=\left\{\begin{array}{ll}
x^{i,j}\left(\int_0^tz^{(i)}_s\,ds\right)\,,\;\;t\ge0\,,\;\;\mbox{if $i\neq j$}\\
x_i+x^{i,i}\left(\int_0^tz^{(i)}_s\,ds\right)\,,\;\;t\ge0\,,\;\;\mbox{if $i=j$}\\
\end{array}\right.\]
admits a unique solution $\{(z^{i,j}_t,\,t\ge0),\;i,j\in[d]\}$ and the system
\[z^{(j)}_t=x_j+\sum_{i=1}^dx^{i,j}\left(\int_0^tz^{(i)}_s\,ds\right)\,,\;\;t\ge0,\,j\in[d]\]
admits a unique solution $(z^{(i)}_t,\,t\ge0,\,i\in[d])$. These solutions are functionals of stopped functions 
$\{(x^{i,j}_t,\,0\le t\le t_x^{(i)}),\;i,j\in[d]\}$.
\end{lemma}
\begin{proof} The proof of the first part of the lemma can be done by applying Lemma \ref{1377} to the discrete time skeleton 
of the functions $\{(x^{i,j}_t,\,t\ge0),\;i,j\in[d]\}$, exactly as for the proof of the first part of Theorem \ref{6803}, see Subsection
\ref{1842}. 

Then the proof of the existence and uniqueness of the solutions of both systems is straightforward. Let $\tau_n$, $n\ge1$ be 
the discrete, ordered sequence of jump times of the  processes $\{(x^{i,j}_t,\,t\ge0),\;i,j\in[d]\}$ (note that two functions $x^{i,j}$ 
can jump simultaneously). Then for each of these  two systems the solution can be constructed in between the times $\tau_n$ 
and $\tau_{n+1}$ in a unique way. 
$\Box$
\end{proof}

\subsection{Proof of Theorem  \ref{7462}.}  Let $(\theta_{k,n},\,k,n\ge1)$ be a family of independent random variables,
such that for each $k$, $\theta_{k,n}$ is uniformly distributed on $[n]$. We assume moreover that the family 
$(\theta_{k,n},\,k,n\ge1)$ is independent of the compound Poisson process $X=(X^{(1)},\dots,X^{(d)})$. 

Then let us construct a multitype branching forest with edge lengths in the following way. Let $\tau^{(i)}_n$, $n\ge1$ 
be the sequence of ordered jump times of the process $X^{(i)}$. We first start with $x_i$ vertices of type 
$i\in[d]$. Let $i$ such that $x_i^{-1}\tau^{(i)}_1=\min\{x_j^{-1}\tau^{(j)}_1:x_j^{-1}\tau^{(j)}_1\le T_x^{(j)},\,j\in[d]\}$. 
Then we construct the branches issued from each vertex of type $j\in[d]$ and, in the scale time of our forest 
in construction, at time $x_i^{-1}\tau_1^{(i)}$, we choose among the $x_i$ vertices of type $i$ according to $\theta_{1,x_i}$ 
the  vertex who gives birth. 

Then the construction is done recursively. Let $y_j$ be the number of vertices of type $j\in[d]$ in the 
forest at time $x_i^{-1}\tau_1^{(i)}$ and let $Y=(Y^{(1)},\dots,Y^{(d)})$ such that $Y^{(j)}$ corresponds to $X^{(j)}$ shifted at 
time $s_j=x_i^{-1}x_j\tau_1^{(i)}$, i.e. $Y^{(j)}_t=X^{(j)}_{s_j+t}$. Then let $\tau^{(j)}_{Y,n}$, $n\ge1$  be 
the sequence of ordered jump times of the process $Y^{(j)}$, and let $k$ such that $y_k^{-1}\tau^{(k)}_{Y,1}=\min\{y_j^{-1}\tau^{(j)}_{Y,1}:x_i^{-1}\tau_1^{(i)}+y_j^{-1}\tau^{(j)}_{Y,1}\le T_x^{(j)},\,j\in[d]\}$. 
Then we continue the construction of the branches issued from each vertex of type $j\in[d]$ and
at time $x_i^{-1}\tau_1^{(i)}+y_k^{-1}\tau^{(k)}_{Y,1}$, we choose among the $y_k$ vertices of type $k$ according to 
$\theta_{2,y_k}$ the vertex who gives birth. This construction is performed until all processes $X^{(i)}$ are stopped
at time $T_x^{(i)}$.

It is clear from this construction that the forest which is obtained is a multitype branching forest with edge lengths, with the 
required distribution and such that the processes $Z^{i,j}$ defined as in Definition \ref{2593} with respect to this forest 
satisfy equation (\ref{3521}).

Finally, the fact that equation, 
\[(Z^{(1)}_t,\dots,Z^{(d)}_t)=x+
\left(\sum_{i=1}^dX^{i,1}_{\int_0^t Z^{(i)}_s\,ds},\dots,\sum_{i=1}^dX^{i,d}_{\int_0^t Z^{(i)}_s\,ds}\right)\,,
\;\;\;t\ge0\]
admits a unique solution is a direct consequence of the second part of Lemma \ref{1592}. $\;\Box$

\newpage

\begin{singlespace} 
\end{singlespace}

\end{document}